\documentclass{article}

\usepackage{arxiv}

\usepackage[utf8]{inputenc} 
\usepackage[T1]{fontenc}    
\usepackage[hypertexnames=true]{hyperref}       
\usepackage{url}            
\usepackage{booktabs}       
\usepackage{amsmath, amsfonts}       
\usepackage{amsthm}
\usepackage{nicefrac}       
\usepackage{microtype}      
\usepackage{cleveref}       
\usepackage{graphicx}
\usepackage{doi}
\usepackage{enumerate}

\usepackage{amssymb}
\usepackage{algorithmic}
\usepackage[ruled,linesnumbered,vlined]{algorithm2e}
\usepackage{array}
\usepackage[caption=false,font=normalsize,labelfont=sf,textfont=sf]{subfig}
\usepackage{textcomp}
\usepackage{stfloats}
\usepackage{verbatim}
\usepackage{cite}
\usepackage{multirow} 
\usepackage{scalerel}
\usepackage[title,titletoc]{appendix}

\theoremstyle{plain}

\newtheorem{lemma}{Lemma}
\newtheorem{proposition}{Proposition}
\theoremstyle{definition}
\newtheorem{assumption}{Assumption}
\newtheorem{definition}{Definition}
\theoremstyle{remark}
\newtheorem{remark}{Remark}

\title{Efficient globalization of heavy-ball type methods for unconstrained optimization based on curve searches}

\author{ 
	Federica Donnini\\
	Global Optimization Laboratory (GOL) \\
	Department of Information Engineering \\
	University of Florence \\
	Via di Santa Marta, 3, 50139, Florence, Italy \\
	\texttt{federica.donnini@unifi.it} \\
	\And
	Matteo Lapucci\\
	Global Optimization Laboratory (GOL) \\
	Department of Information Engineering \\
	University of Florence \\
	Via di Santa Marta, 3, 50139, Florence, Italy \\
	\texttt{matteo.lapucci@unifi.it} \\
	\And
	Pierluigi Mansueto \\
	Global Optimization Laboratory (GOL) \\
	Department of Information Engineering \\
	University of Florence \\
	Via di Santa Marta, 3, 50139, Florence, Italy \\
	\texttt{pierluigi.mansueto@unifi.it} \\
}

\renewcommand{\headeright}{Federica Donnini, Matteo Lapucci and Pierluigi Mansueto}
\renewcommand{\undertitle}{}
\renewcommand{\shorttitle}{}

\hypersetup{
pdftitle={Efficient globalization of heavy-ball type methods for unconstrained optimization based on curve searches},
pdfsubject={},
pdfauthor={Federica Donnini, Matteo Lapucci, Pierluigi Mansueto},
pdfkeywords={Curve search, Global convergence, Complexity bounds, Heavy-ball method},
}

\begin{document}
\maketitle

\begin{abstract}
	In this work, we deal with unconstrained nonlinear optimization problems. Specifically, we are interested in methods carrying out updates possibly along directions not of descent, like Polyak's heavy-ball algorithm. Instead of enforcing convergence properties through line searches and  modifications of search direction when suitable safeguards are not satisfied, we propose a strategy based on searches along curve paths: a curve search starting from the first tentative update allows to smoothly revert towards a gradient-related direction if a sufficient decrease condition is not met. The resulting algorithm provably possesses global convergence guarantees, even with a nonmonotone decrease condition. While the presented framework is rather general, particularly of interest is the case of parabolic searches; in this case, under reasonable assumptions, the resulting algorithm can be shown to possess optimal worst case complexity bounds for reaching approximate stationarity in nonconvex settings. 
	Practically, we show that the proposed globalization strategy allows to consistently accept (optimal) pure heavy-ball steps in the strongly convex case, while standard globalization approaches would at times negate them before even evaluating the objective function.  
	Preliminary computational experiments also suggest that the proposed framework might be more convenient than classical safeguard based approaches. 
\end{abstract}

\MSCs{90C26 \and 90C30 \and 65Y20}

\section{Introduction}
In this paper, we focus on the problem  of minimizing a continuously differentiable function over the Euclidean space, i.e., we consider optimization problems of the form
\begin{equation}
	\label{OP}
	\min_{x\in\mathbb{R}^n}f(x),
\end{equation}
where $f:\mathbb{R}^n\to \mathbb{R}$ is continuously differentiable.
The literature on algorithms to tackle this general class of problems is extremely vast. We refer the reader to monographies like \cite{bertsekas1999nonlinear,grippo2023introduction} for a thorough overview of the main methodological frameworks.

Here, we develop our discussion starting from the famous Polyak's heavy-ball method \cite{polyak1964some,polyak1987introduction}. The main idea of Polyak's algorithm consists in adding a momentum term, with a fixed weight, to the standard gradient descent update; iterations then take the form
$$x^{k+1}=x^k-\alpha \nabla f(x^k)+\beta (x^k-x^{k-1}),$$
where $\alpha, \beta \in \mathbb{R}$ are typically fixed positive values.
Partially repeating the previous step has the effect of controlling oscillation and providing acceleration in low curvature regions. All of this can, in principle, be achieved using only information already available: no additional function evaluations are required to be carried out. This feature makes the addition of momentum terms appealing in large-scale settings.

Theoretically, heavy-ball method is proven to have a sound behavior and to provide a substantial improvement w.r.t.\ gradient descent under regularity assumptions and for suitable choices of the parameters $\alpha$ and $\beta$. In particular, local linear convergence (with better constants than both gradient descent and Nesterov's method) can be guaranteed if the objective function is $\mu$-strongly convex, twice differentiable and $L$-smooth, as long as the two parameters belong to suitable intervals depending on the values of $L$ and $\mu$ \cite{ghadimi2015global,lessard2016analysis,polyak1987introduction}. Global convergence with fast linear rate can even be proved for strongly convex quadratic problems \cite{ghadimi2015global,polyak1964some,saab2022adaptive}. In both cases, ``optimal values'' for $\alpha$ and $\beta$ exist, again depending on $\mu$ and $L$, leading to the best constant in the convergence rate \cite{polyak1964some}. However, the convergence properties are already lost if some of the aforementioned assumptions are relaxed, see for instance \cite[Section 4.6]{lessard2016analysis}. 

In absence of convexity guarantees, fast linear convergence can still be proved under the Polyak-Lojasiewicz assumption \cite{kassing2024polyak}. In the general nonconvex case, however, no convergence results are available for the ``pure'' heavy-ball method. In this setting, one of the main issues lies in the choice of the stepsize parameters $\alpha$ and $\beta$, since constant values across all iterations can hardly work in such a scenario; dynamic choices shall then be considered. However, most existing works on this issue either introduce safeguards concerning the descent property of the heavy-ball direction \cite{fan2023msl} or are in fact closely related to conjugate gradient approaches \cite{lapucci2024globallyconvergentgradientmethod,Lee17,Liu2024,Powell1977,Tang24,zhang2023drsomdimensionreducedsecondorder}; in both cases, we will detailedly argue that the resulting methods cannot be properly seen as a heavy-ball algorithm.

The main contribution of this work concerns a novel globalization strategy for the heavy-ball method - and actually for any iterative method - in the nonconvex scenario. The proposed solution scheme revolves around a curve search strategy \cite{ben1990curved,botsaris1978differential,goldfarb1980curvilinear,gould2000exploiting,shi2005new,xu2016global}. Specifically, we prove soundness and global convergence results for any iterative method that carries out backtracking searches along suitable search curves with the aim of satisfying a (possibly nonmonotone) Armijo-type sufficient decrease condition. Global convergence of the resulting general framework is proved under a set of assumptions on the curves that differ from those considered in other related works and, to the best of our knowledge, the convergence result for a nomonotone curve search approach is novel in the literature. Furthermore, we show that, in the particular case of quadratic curves and under reasonable assumptions, optimal worst-case complexity bounds can be proved for the proposed framework in the nonconvex case. Complexity results for curve search based methods constitute, to the best of our knowledge, another original contribution to the literature.

We then turn back to the specific case of heavy-ball: not only we outline that the algorithm fits the assumption of the general curvilinear search framework, but we also show that, differently than other globalization strategies from the literature, our technique allows to recover the behavior of the pure heavy-ball method in the strongly convex case. Finally, by quite extensive numerical experiments on nonconvex problems, we highlight that the globalization strategy is effective in practice and more efficient than available alternatives. 

The remainder of the paper is organized as follows. In Section \ref{sec:gen_cs}, we describe and analyze our novel (nonmonotone) descent method based on curve searches from a theoretical perspective. Section \ref{sec::sqc} considers a special case of our framework involving quadratic curves, focusing primarily on its worst-case complexity properties. In Section \ref{sec::heavyball}, we examine the heavy-ball method, highlighting the differences between our curve search approach and other globalization strategies for Polyak's algorithm known in the literature. Section \ref{sec::comp-exp} presents numerical experiments on both strongly convex and nonconvex problems, comparing our method with standard strategies for problem \eqref{OP}. Finally, in Section \ref{sec::conclusions}, we provide concluding remarks.

\section{A two-direction descent method based on curvilinear searches}
\label{sec:gen_cs}
Of the numerous methods developed to tackle problems of the form \eqref{OP}, many of the most efficient ones fall within the class of line-search-based methods
\cite{grippo2023introduction}, i.e., approaches that rely on update rules of the form
\begin{equation*}
	x^{k+1}=x^k+\alpha_k d_k,
\end{equation*}
where $d_k\in\mathbb{R}^n$ is a search direction (typically a descent direction) and $\alpha_k\in~\mathbb{R}$ is a positive stepsize, usually selected via a backtracking procedure. The direction and the stepsize are generally chosen so that the resulting sequence $\{x^k\}$ has guarantees of convergence towards stationarity for \eqref{OP}.

By a closer look, line-search methods can be seen as a special case of curve-search methods \cite{ben1990curved,goldfarb1980curvilinear}. In general, a curve search method iteratively defines the new iterate according to
\begin{equation}\label{eq::CS}
	x^{k+1}=\gamma_k(t_k),
\end{equation}
where $\gamma_k:[0,1]\to\mathbb{R}^n$ is a differentiable curve and $t_k$ is a positive parameter that indicates how far to move along the curve.
Clearly, for a line-search type algorithm, $\gamma_k$ would define a straight line.

For the purpose of this work, we will focus on search curves whose initial velocity is suitably related to the gradient. Formally, we consider curves \text{$\gamma_k(t):[0,1]\to\mathbb{R}^n$} such that
\begin{equation*}
	\gamma_k(t) = \gamma(t;x^k,d_k,\xi_k),
\end{equation*}
by which we mean that each $\gamma_k$ belongs to the same family of curves $\gamma$ whose parametrization is dependent on the current solution  $x^k\in\mathbb{R}^n$, a suitable search direction $d_k\in\mathbb{R}^n$ and some other quantities $ \xi_k\in\mathbb{R}^p$ (which, for instance, might denote a second search direction). 
For these curves, we will make throughout the paper the following set of assumptions.
\begin{assumption}\label{ass::defcurve}
	The following conditions hold for the family of search curves $\gamma(\cdot;x,d,\xi)$: 
	\begin{enumerate}[(a)]
		\item $\gamma(t;x,d,\xi)$ is continuous w.r.t.\ $t,x,d$ and $\xi$;
		\item $\gamma(0;x,d,\xi)=x$, i.e., the starting point of the curve is the solution $x$;
		\item the velocity $$\gamma'(t;x,d,\xi) = \left(\frac{\partial \gamma_1(t;x,d,\xi)}{\partial t},\ldots,\frac{\partial \gamma_n(t;x,d,\xi)}{\partial t}\right)^\top$$ exists and is continuous for all $t,x,d$ and $\xi$;
		\item $\gamma'(0;x,d,\xi)=d$, i.e., the initial velocity of the curve is given by direction $d$.
	\end{enumerate}
\end{assumption}

We now give a fundamental definition.
\begin{definition}
	Let $f:\mathbb{R}^n\to\mathbb{R}$ be a continuously differentiable function. We say that $\gamma_k$ is a \emph{descent curve} for $f$ at $x^k\in\mathbb{R}^n$ if $\gamma'_k(0)$ is a descent direction for $f$ at $x^k$, i.e., $\nabla f(x^k)^\top \gamma_k'(0)<0$.    
\end{definition}

For the ease of notation, we also define the function $\varphi_k:[0,1]\to\mathbb{R}$ such that $\varphi_k(t) = f(\gamma_k(t))$. We shall note that $\varphi_k'(t)=\nabla f(\gamma_k(t))^\top\gamma_k'(t)$
and $\varphi_k'(0)=\nabla f(\gamma_k(0))^\top\gamma_k'(0)=\nabla f(x^k)^\top d_k.$
Hence, $\gamma_k$ is a descent curve for $f$ in $x^k$ if and only if $\varphi'_k(0)<0$.

Before turning to the specific definition and the analysis of the algorithmic framework, we give another useful definition, taking inspiration from the well-known concept of gradient-related direction \cite{cartis2022evaluation}.

\begin{definition}
	Let $\{x^k\}$ be the sequence of points generated by an iterative algorithm of the form \eqref{eq::CS}, with $\{\gamma_k\}$ being the corresponding sequence of search curves. We say that the latter is a sequence of curves \emph{gradient-related} to $\{x^k\}$ if the sequence of directions $\{\gamma_k'(0)\}$ is gradient-related to $\{x^k\}$, i.e., if there exist $c_1,c_2>0$ such that for all $k\geq0$ the following conditions hold:
	\begin{equation*}
		\nabla f(x^k)^\top \gamma_k'(0)\leq - c_1\|\nabla f(x^k)\|^2, \qquad \|\gamma_k'(0)\|\leq c_2\|\nabla f(x^k)\|.
	\end{equation*}    
\end{definition}

\subsection{Curve-searches}
While we have characterized the update rule of curve search algorithms \eqref{eq::CS}, we have yet to specify how to select a suitable value of $t_k$. Various options could be considered: an exhaustive review of such methods can be found in \cite{xu2016global}, which demonstrates how step-size selection techniques used in line-search methods can be easily extended to define corresponding curve-search strategies.

Here, we will focus on the \emph{Armijo-type Curve Search Rule} \cite{xu2016global}, which mimics the structure of the Armijo-type line search and requires the parameter $t_k\geq0$ to satisfy
$\varphi_k(t_k)\le \varphi_k(0)+\sigma t_k\varphi'_k(0),$
or, in other words, 
\begin{equation}\label{eq::ARcond}
	f(\gamma_k(t_k))\leq f(x^k)+\sigma t_k\nabla f(x^k)^\top d_k,
\end{equation}
where $\sigma\in(0,1)$.

The backtracking curve search procedure is described by the simple pseudocode reported in Algorithm \ref{alg:ACS}.

\begin{algorithm}
	\SetNoFillComment
	\caption{Armijo-type Curve Search}
	\label{alg:ACS}
	\KwData{$x^k\in\mathbb{R}^n,d_k\in\mathbb{R}^n,\ \gamma_k:[0,1]\to\mathbb{R}^n,\ \Delta_0 \in (0, 1],\ \sigma\in(0,1),\ \delta \in(0,1)$.}
	
	Set $t= \Delta_0$   
	
	\While{$f(\gamma_k(t))>f(x^k)+\sigma t \nabla f(x^k)^\top d_k$}{
		Set $t = \delta t$
	}
	Set $t_k=t $
	\\ \textbf{return} $t_k$
\end{algorithm}

We can immediately state a finite termination result for Algorithm \ref{alg:ACS}.
\begin{proposition}\label{prop::ACSprop}
	Let $x^k \in \mathbb{R}^n$ and $d_k \in \mathbb{R}^n$ such that $\nabla f(x^k)^\top d_k < 0$. Let $\gamma_k:[0,1]\to\mathbb{R}^n$ be a descent curve for $f$ in $x^k$ such that $\gamma_k(t) = \gamma(t;x^k,d_k,\xi_k)$ and Assumption \ref{ass::defcurve} is satisfied. Then, the Armijo-type curve search procedure (Algorithm \ref{alg:ACS}) terminates in a finite number of iterations, returning a stepsize $t_k > 0$ which satisfies equation \eqref{eq::ARcond}. Moreover, one of the two following conditions holds:
	\begin{enumerate}[(i)]
		\item $t_k=\Delta_0$;
		\item $t_k<\Delta_0$ and $f(\gamma_k(\frac{t_k}{\delta}))>f(x^k)+\sigma \frac{t_k}{\delta}\nabla f(x^k)^\top d_k.$
	\end{enumerate}
\end{proposition}
\begin{proof}
	Following the proof of \cite[Lemma 2.2]{xu2016global}, we prove that the Armijo-type curve search procedure provides in a finite number of iterations a stepsize $t_k > 0$ satisfying the Armijo-type condition \eqref{eq::ARcond}. If no backtrack is needed, condition (i) straightforwardly holds; if, on the other hand, at least a backtrack step is carried out, then the stepsize $t_k/\delta$ is the last tested step not satisfying equation \eqref{eq::ARcond}; hence, condition (ii) holds in this case.
\end{proof}

\subsection{Convergence analysis}
\label{sec:conv_gen}

In this section, we state and prove the global convergence results for curve-search based algorithms. The convergence property under discussion is analogous to the one proved in \cite{xu2016global}, but is derived under different assumptions on the curves: indeed, we are going to show that an iterative method that follows the update \eqref{eq::CS}, with $t_k$ chosen so as to satisfy condition \eqref{eq::ARcond}, possesses global convergence guarantees if the descent curve $\gamma_k=\gamma(t;x^k,d_k,\xi_k)$ is such that $\{d_k\}$ is gradient-related and $\{\xi_k\}$ is bounded.

To prove this result we need to state some further, reasonable assumptions. 
\begin{assumption}
	\label{ass::L0}
	Let $x^0 \in \mathbb{R}^n$. The level set $\mathcal{L}_0 = \{x \in \mathbb{R}^n\mid f(x) \le f(x^0)\}$ is compact.
\end{assumption}

\begin{assumption}
	\label{ass::xi}
	Let $\{(x^k, \gamma(\cdot; x^k, d_k, \xi_k))\}$ be the sequence generated by an iterative algorithm
	of the form \eqref{eq::CS}. The sequence $\{\xi_k\}$ is bounded, i.e., $\exists C > 0$ such that $\|\xi_k\| \le C$ for all $k$.
\end{assumption}

\begin{proposition}\label{prop:globalconvm}
	Let Assumption \ref{ass::L0} hold. Let $\{(x^k,\gamma_k(\cdot;x^k,d_k,\xi_k))\}$ be the sequence generated by an iterative algorithm of the form \eqref{eq::CS} such that: 
	\begin{enumerate}[i)]
		\item stepsizes $t_k$ are chosen by the Armijo-type curve search procedure (Algorithm \ref{alg:ACS}), for all $k$;
		\item $\gamma_k(t) = \gamma(t;x^k,d_k,\xi_k)$ satisfies Assumption \ref{ass::defcurve};
		\item the sequence of curves $\{\gamma_k\}$ is gradient-related to $\{x^k\}$;
		\item Assumption \ref{ass::xi} holds for sequence $\{\xi_k\}$. 
	\end{enumerate}  Then, the sequence $\{x^k\}$ admits accumulation points and each accumulation point $\bar{x}$ is stationary, i.e., $\nabla f(\bar{x})=0$.
\end{proposition}
\begin{proof}
	For all $k$, we have, by the Armijo condition \eqref{eq::ARcond}, that 
	\begin{align*}
		f(x^{k+1})=  f(\gamma_k(t_k)) &\leq f(x^k) + \sigma t_k \nabla f(x^k)^\top d_k \\&\leq f(x^k) -\sigma c_1 t_k \|\nabla  
		f(x^k)\|^2<f(x^k),
	\end{align*} where the second inequality follows from the sequence $\{\gamma_k\}$ being gradient-related to $\{x^k\}$, and the last inequality from the non-negativeness of the norm.
	Thus, the sequence $\{f(x^k)\}$ is monotone decreasing and, by the compactness of $\mathcal{L}_0$ (Assumption \ref{ass::L0}), it follows that the sequence $\{x^k\}$ admits accumulation points, each one belonging to $\mathcal{L}_0$. Furthermore, the sequence $\{f(x^k)\}$ converges to some finite value $f^\star$ and 
	\begin{equation}\label{limfk}
		\lim_{k\to\infty} f(x^{k+1})-f(x^k)=0.
	\end{equation}
	
	Let $\bar{x}$ be one of the accumulation points, that is, there exists a sub-sequence $K\subseteq\{0,1,\dots\}$ such that $\lim_{k\in K,k\to\infty} x^k=\bar{x}$. Let us assume by contradiction that $\bar{x}$ is not stationary, i.e., $\|\nabla f(\bar{x})\| = \nu >0$.
	
	Again by the Armijo condition \eqref{eq::ARcond} and the sequence $\{\gamma_k\}$ being gradient-related to $\{x^k\}$, we obtain $$f(x^{k+1})-f(x^k)\leq -\sigma c_1t_k\|\nabla f(x^k)\|^2.$$
	Since the previous equation holds for every $k$, we can take the limit for $k\in K$ and $k\to\infty$ recalling equation \eqref{limfk}: $$0\leq\lim_{k\in K,k\to\infty} -\sigma c_1t_k\|\nabla f(x^k)\|^2,$$ and, thus, since $\sigma \in (0, 1)$, $c_1 > 0$ and the norm operator is non-negative, we get 
	\begin{equation*}
		\lim_{k\in K,k\to\infty} t_k\|\nabla f(x^k)\| = 0.
	\end{equation*}
	By the continuity of $\nabla f$ we have 
	\begin{equation}
		\label{eq::lim-grad}
		\lim_{k\in K,k\to\infty}\|\nabla f(x^k)\|=\|\nabla f(\bar x)\|=\nu>0.
	\end{equation}
	Hence, it must be that $\lim_{k\in K,k\to \infty}t_k=0.$
	
	The last limit implies that there exists $\bar k\in K$ such that $t_k<\Delta_0$ for all $k\in K$ and $k\geq\bar{k}.$ Then, by Proposition \ref{prop::ACSprop} we have $$f\left(\gamma_k\left(\frac{t_k}{\delta}\right)\right) > f(x^k) + \sigma\frac{t_k}{\delta}\nabla f(x^k)^\top d_k$$ for all $k\in K,\ k\geq \bar{k}$.
	On the other hand, by the Mean Value Theorem we can write, for all $k$,
	\begin{equation*}
		f\left(\gamma_k\left(\frac{t_k}{\delta}\right)\right) = f(\gamma_k(0)) + \frac{t_k}{\delta}\nabla f\left(\gamma_k\left(\theta_k\frac{t_k}{\delta}\right)\right)^\top\gamma_k'\left(\theta_k\frac{t_k}{\delta}\right),
	\end{equation*}
	with $\theta_k\in(0,1)$. Combining the last two results, we get, for all $k\in K$ and $k \ge \bar{k}$,
	\begin{equation}
		\label{eq::bef-lim}
		\nabla f\left(\gamma_k\left(\theta_k\frac{t_k}{\delta}\right)\right)^\top\gamma_k'\left(\theta_k\frac{t_k}{\delta}\right)>\sigma \nabla f(x^k)^\top d_k.
	\end{equation}
	By the gradient-related conditions, we have, for all $k$, $\|\gamma'_k(0)\| = \|d_k\|\leq c_2\|\nabla f(x^k)\|$.
	Recalling now equation \eqref{eq::lim-grad}, we get that the sequence $\{d_k\}_{k \in K}$ is bounded: there exists $K_1\subseteq K$ such that $\lim_{k\in K_1, k\to \infty} d_k=\bar{d}$. Moreover, by Assumption \ref{ass::xi}, the sequence $\{\xi_k\}$ is also bounded; thus, there exists $K_2\subseteq K_1$ such that $\lim_{k\in K_2,k\to\infty}\xi_k=\bar{\xi}$. By the continuity of $\gamma_k$ w.r.t.\ $t,x,d$ and $\xi$ (Assumption \ref{ass::defcurve}a), we then have
	\begin{equation*}
		\lim_{k\in K_2,k\to\infty}\gamma_k(t) = \lim_{k\in K_2,k\to\infty}\gamma(t;x^k,d_k,\xi_k) = \gamma(t;\bar{x},\bar{d},\bar{\xi}) = \bar{\gamma}(t).
	\end{equation*}
	Similarly, by the continuity of $\gamma_k'$ (Assumption \ref{ass::defcurve}c), we can conclude that $$\lim_{k\in K_2,k\to\infty}\gamma'_k(t) = \gamma'(t;\bar{x},\bar{d},\bar{\xi}) = \bar{\gamma}'(t).$$
	
	Now, taking the limit in \eqref{eq::bef-lim} for $k\in K_2,\ k\to\infty,$ recalling the two previous limits, the continuity of $\nabla f$, the boundedness of $\{d_k\}$, that $t_k\to_{K} 0$ and $\theta_k\in(0,1)$, we get that 
	\begin{equation*}
		\nabla f(\bar{\gamma}(0))^\top\bar{\gamma}'(0)= \nabla f(\bar{x})^\top\bar{d} \geq \sigma\nabla f(\bar{x})^\top \bar{d},
	\end{equation*}
	that is, $(1-\sigma)\nabla f(\bar{x})^\top \bar{d}\geq0$
	which implies $\nabla f(\bar x)^\top \bar d\geq0.$ 
	
	However, by the gradient-related conditions and equation \eqref{eq::lim-grad}, we know that
	\begin{equation*}
		\nabla f(\bar{x})^\top \bar{d} = \lim_{k\in K_2,k\to\infty}\nabla f(x^k)^\top d_k\leq \lim_{k\in K_2,k\to\infty} -c_1\|\nabla f(x^k)\|^2 =-c_1\nu^2<0,
	\end{equation*}
	which is a contradiction. We thus get the thesis.
\end{proof}

\begin{remark}
	In \cite{xu2016global}, convergence results for curve search methods were provided under the following assumptions: the curve search sequence $\{\gamma_k\}$ shall satisfy for all $k$
	\begin{equation}\label{eq::xuhp1}
		\|\gamma_k(t)-\gamma_k(0)\|\leq y_1(t)\qquad \text{ and } \qquad \frac{\|\gamma_k'(t)-\gamma_k'(0)\|}{\|\gamma_k'(0)\|}\leq y_2(t),
	\end{equation}
	where $y_1,y_2:\mathbb{R}_+\to\mathbb{R}_+$ are forcing functions - i.e., $\lim_{t\to0}y_i(t)=0,\ i=~1,2$ - independent on $k$.
	These assumptions are substantially different from the assumptions made in this work, as we will also underline later in the manuscript.
\end{remark}

\subsubsection{The nonmonotone case}
\label{sec:conv_nonm}

In this section, we deal again with an iterative process of the form \eqref{eq::CS}, but now the stepsize $t_k$ is chosen with a nonmonotone version of the Armijo-type curve search procedure. In other words, we aim to find a stepsize with Algorithm \ref{alg:ACS} but the termination condition asks $t_k$ to satisfy
\begin{equation}
	\label{eq::nm_ARcond}
	f(\gamma_k(t_k)) \le \max_{0\leq j \leq m(k)}f(x^{k-j})+\sigma t_k \nabla f(x^k)^\top d_k,
\end{equation}
where $m(0)=0$, $0\leq m(k)\leq\min\{m(k-1)+1,M\}$ for $k \ge 1$, and $M \in \mathbb{N}$.
This nonmonotone rule can be seen as the extension to curve search methods of the classical nonmonotone condition proposed in \cite{grippo1986nonmonotone}.

The backtracking procedure can be characterized by the following properties.
\begin{proposition}
	Let $x^k \in \mathbb{R}^n$ and $d_k \in \mathbb{R}^n$ such that $\nabla f(x^k)^\top d_k < 0$. Let $\gamma_k:[0,1]\to\mathbb{R}^n$ be a descent curve for $f$ in $x^k$ such that $\gamma_k(t) = \gamma(t;x^k,d_k,\xi_k)$ and Assumption \ref{ass::defcurve} is satisfied. Then, the nonmonotone Armijo-type curve search procedure terminates in a finite number of iterations, returning a stepsize $t_k > 0$ which satisfies equation \eqref{eq::nm_ARcond}. Moreover, one of the two following conditions holds:
	\begin{enumerate}
		\item $t_k=\Delta_0$;
		\item $t_k<\Delta_0$ and $f(\gamma_k(\frac{t_k}{\delta}))>\max_{0\leq j \leq m(k)}f(x^{k-j})+\sigma \frac{t_k}{\delta}\nabla f(x^k)^\top d_k.$
	\end{enumerate}
\end{proposition}
\begin{proof}
	By Proposition \ref{prop::ACSprop} and the fact that $\max_{0\leq j \leq m(k)}f(x^{k-j}) \ge f(x^k)$, we get that the non-monotone Armijo-type curve search procedure finds, in a finite number of steps, a stepsize $t_k > 0$ such that $f(\gamma_k(t_k)) \le f(x^{k})+\sigma t_k \nabla f(x^k)^\top d_k \le \max_{0\leq j \leq m(k)}f(x^{k-j})+\sigma t_k \nabla f(x^k)^\top d_k$. Following the last part of the proof of Proposition \ref{prop::ACSprop}, we straightforwardly get the thesis.
\end{proof}

\begin{proposition}
	\label{prop::nm_globalconvm}
	Let Assumption \ref{ass::L0} hold. Let $\{(x^k,\gamma(\cdot;x^k,d_k,\xi_k))\}$ be the sequence generated by an iterative algorithm of the form (\ref{eq::CS}) such that stepsizes $t_k$ are chosen by the nonmonotone Armijo-type curve search procedure for all $k$; furthermore assume that hypotheses ii)-iii)-iv) of Proposition \ref{prop:globalconvm} hold. Then, the sequence $\{x^k\}$ admits accumulation points and there exists at least one accumulation point $\bar{x}$ which is stationary, i.e., $\nabla f(\bar{x})=0$.
\end{proposition}
\begin{proof}
	Let us define $l(k)\in\mathbb{N}$ such that $k-m(k)\leq l(k)\leq k$ and $f(x^{l(k)})=\max_{0\leq j \leq m(k)}f(x^{k-j})$.
	Given $m(k+1)\leq m(k)+1$, we get by the Armijo condition \eqref{eq::nm_ARcond} that the sequence $\{f(x^{l(k)})\}$ is non-increasing: 
	\begin{gather*}
		\begin{aligned}
			f(x^{l(k+1)})&=\max_{0\leq j \leq m(k+1)} f(x^{k+1-j})\leq \max_{0\leq j\leq m(k)+1}f(x^{k+1-j})\\&=\max \left\{f(x^{k+1}),\max_{0\leq j \leq m(k)}f(x^{k-j})\right\}=\max\{f(x^{k+1}),f(x^{l(k)})\}=f(x^{l(k)}).
		\end{aligned}
	\end{gather*}
	By the compactness of $\mathcal{L}_0$ (Assumption \ref{ass::L0}) and the continuity of $f$, it thus follows that the sequence $\{f(x^{l(k)})\}$ converges to some finite value $f^\star$.
	
	Now, by the Armijo-type condition \eqref{eq::nm_ARcond} and the gradient-related conditions, it follows for $k>M$ that
	\begin{equation}
		\label{eq::for_app}
		\begin{aligned}
			f(x^{l(k)})&\leq\max_{0\leq j \leq m(l(k)-1)}f(x^{l(k)-1-j})+\sigma t_{l(k)-1}\nabla f(x^{l(k)-1})^\top d_{l(k)-1}\\
			& \le f(x^{l(l(k)-1)})-\sigma c_1t_{l(k)-1}\|\nabla f(x^{l(k)-1})\|^2.
		\end{aligned}
	\end{equation}
	The inequality holds for every $k$; hence, we can take the limit for $k\to\infty$ recalling $\lim_{k\to\infty}f(x^{l(k)})=f^\star$, $\sigma \in (0, 1)$, $c_1 > 0$ and the non-negativeness of the norm, and we get 
	\begin{equation}
		\label{eq::limlk_1}
		\lim_{k\to\infty}t_{l(k)-1}\|\nabla f(x^{l(k)-1})\|=0.
	\end{equation}
	
	Let us assume, by contradiction, that any of the accumulation points of $\{x^{l(k)-1}\}$ is nonstationary; in particular, let $K \subseteq \{0, 1, \ldots\}$ such that $\lim_{k \in K, k \to \infty} x^{l(k)-1} = \bar{x}$, with $\bar{x}$ nonstationary, that is, $\|\nabla f(\bar{x})\| = \nu > 0$. Hence, by \eqref{eq::limlk_1} it must be that $\lim_{k \in K, k \to \infty}t_{l(k)-1} = 0$. Following now a reasoning similar to the one in the last part of the proof of Proposition \ref{prop:globalconvm}, we get a contradiction and, therefore, all limit points of $\{x^{l(k)-1}\}_K$ are stationary.
\end{proof}

Proposition \ref{prop::nm_globalconvm} is interesting, as it provides us with a global convergence result for a nonmonotone algorithm even in absence of the property that the distance between two iterates goes to zero. While weaker, the proposition is enough to guarantee that we can obtain a stationary point with arbitrary accuracy in finite time.

To strengthen the result -- ensuring, as in the monotone case and the nonmonotone line search case of \cite{grippo1986nonmonotone}, that every accumulation point of the iterate sequence is stationary -- we could impose a further assumption on the sequence of curves, similar to the first condition in \eqref{eq::xuhp1}. The result is formalized in the following proposition. The proof is deferred to the Appendix, as it is can be straightforwardly obtained combining the ones of Proposition \ref{prop::nm_globalconvm} and of the Theorem in \cite{grippo1986nonmonotone}.

\begin{proposition}
	\label{prop::nm_globalconvm_strong}
	Let $\{(x^k,\gamma(\cdot;x^k,d_k,\xi_k))\}$ be the sequence generated by an iterative algorithm of the form (\ref{eq::CS}) such that all the assumptions and hypotheses of Proposition \ref{prop::nm_globalconvm} hold. Let us further assume that the sequence of curves $\{\gamma_k\}$ satisfies, for all $k$, the condition $\|\gamma_k(t) - \gamma_k(0)\| \le y(t\|d_k\|)$, with $y: \mathbb{R}_+\rightarrow\mathbb{R}_+$ being a forcing function. Then, the sequence $\{x^k\}$ admits accumulation points and each accumulation point $\bar{x}$ is stationary, i.e., $\nabla f(\bar{x})=0$.
\end{proposition}
\begin{proof}
	See Appendix \ref{app::proof}.    
\end{proof}

\section{Searching along quadratic curves}
\label{sec::sqc}

The results we have presented in Section \ref{sec:gen_cs} hold for any curve $\gamma_k$ which satisfies Assumption \ref{ass::defcurve}. From now on, we will consider a specific family of curves that satisfies said assumption. In particular we will consider polynomial curves of degree two that can be written as
\begin{equation}\label{eq::gengamma2}
	\gamma(t;x,d,s)=x+td+t^2(s-d),
\end{equation}
where we will be setting $x=x^k$ from the sequence of iterates $\{x^k\}$, $d=d_k$ from  a gradient-related sequence of directions $\{d_k\}$ and $s=s_k\in\mathbb{R}^n$ from some bounded sequence of directions $\{s_k\}$. 

Again, we will be employing the following notation to ease readability:
\begin{equation*}
	\gamma_k(t)=\gamma(t;x^k,d_k,s_k).
\end{equation*}
Note that $\gamma_k(1)=x^k+d_k+(s_k-d_k)=x^k+s_k$, which means that the first tentative solution visited during a curve search (Algorithm \ref{alg:ACS}) with unit initial value for $t$ will always be the point $x^k+s_k$. 

\begin{remark}
	\label{rem::bc}
	We get some interesting insights if we change the basis in \eqref{eq::gengamma2} from the canonical one to Bernstein's, i.e., if we construct the curve as a Bézier curve of degree 2 \cite{farin2000essentials} generated by the set of control points $P_0 = x,\ P_1=x+\frac{1}{2}d$ and $P_2=x+s$ (see also Fig.\ \ref{fig:enter-label}): 
	\begin{equation*}\label{BC2}
		\gamma (t) = (1-t)^2 P_0 + 2t(1-t)P_1 + t^2P_2,\quad t\in[0,1].
	\end{equation*}
	The geometric properties of Bézier curves and their control points enhance interpretability, as the quadratic Bézier curve always passes through the endpoints $P_0$ and $P_2$, has initial velocity $2(P_1-P_0)$ and is contained in the convex hull of $P_0,P_1,P_2$. This implies that, by fixing the starting point $P_0$ and the end point $P_2$, one can control the shape of the curve by moving the intermediate point $P_1$ along the direction $d$, i.e., by varying the module of the initial velocity we steer the curve closer to either direction $d$ or $s$.
\end{remark}

\begin{figure}[h]
	\centering
	\includegraphics[width=0.6\textwidth]{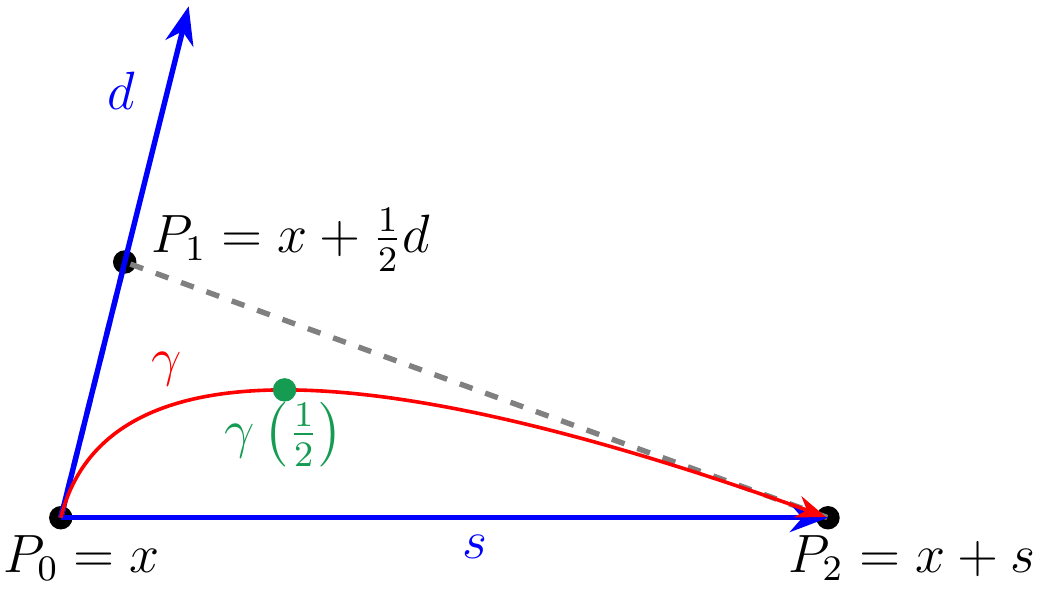}
	\caption{Quadratic search curve starting at $x$, ending at $x+s$ and with initial velocity $d$. The points $P_0=x$, $P_1 = x+\frac{1}{2}d$ and $P_2=x+s$ are the control points of Bézier's expression of the curve. } 
	\label{fig:enter-label}
\end{figure}

The family of curves of the form \eqref{eq::gengamma2} can induce different curve search methods, depending on the choices of $d_k$ and $s_k$ with respect to $x^k$.
The choice of $d_k$ arguably reduces to the negative gradient or, more in general, to a gradient-related direction, as that is the direction that drives convergence results; on the other hand, possible choices for $s_k\in\mathbb{R}^n$ include:
\begin{itemize}
	\item any direction such that $s_k-d_k$ is a nonascent direction of negative curvature, similarly as what is done in \cite{goldfarb1980curvilinear};
	\item Newton's direction $s_k=-[\nabla^2f(x^k)]^{-1}\nabla f(x^k)$;
	\item momentum term $s_k=x^{k}-x^{k-1}$;
	\item heavy-ball direction $s_k = -\alpha\nabla f(x^k)+ \beta(x^k-x^{k-1})$.
\end{itemize}
All these choices for $s_k$ share the property that the direction might not be a descent direction and, thus, line searches along these directions may not be sufficient to ensure the convergence of the algorithm.

In the next section, we will refine the general theoretical analysis already provided in Section \ref{sec:conv_gen} for this particular setting, with special focus on complexity results that, to the best of our knowledge, represent a first in the literature of curve search methods. 

\subsection{Theoretical analysis}
We start the formal analysis showing that curves of the form \eqref{eq::gengamma2} satisfy the assumptions of the more general case.
\begin{lemma}
	Curves of the form \eqref{eq::gengamma2} satisfy Assumption \ref{ass::defcurve}.
\end{lemma}
\begin{proof}
	Continuity of $\gamma(t;x,d,s)$  w.r.t.\ all its arguments can straightforwardly be inferred from its definition in \eqref{eq::gengamma2}. Condition $\gamma(0;x,d,s)=x$ also trivially holds by definition. The derivative w.r.t.\ $t$ always exists and is given by
	$\gamma'(t;x,d,s)=d+2(s-d)t$, which, in turn, is continuous for all the parameters under consideration; moreover, $\gamma'(0;x,d,s)=d$ always holds. Thus, the curve satisfies the conditions in Assumption \ref{ass::defcurve}. 
\end{proof}

We therefore work with a family of curves such that $\gamma_k(0)=x^k,\ \gamma_k'(0)=d_k$ and $\gamma_k(1)=x^k+s_k$. Interestingly, all the assumptions required for Proposition \ref{prop:globalconvm} are satisfied by this particular choice of the curves, as $d_k$ is gradient-related and we are assuming that $\xi_k=s_k$ is bounded. The global convergence result is thus immediately retrieved.

\begin{remark}
	We shall observe that curves of the form \eqref{eq::gengamma2} do not necessarily satisfy assumptions of the form \eqref{eq::xuhp1}, so that convergence results obtained for this class of curve search algorithms are indeed novel to the literature.
	
	First, if we assumed that both $\{d_k\}$ and $\{s_k\}$ are bounded (i.e., for all $k$, $\|d_k\| \le M_1$ and $\|s_k\| \le M_2$, with $M_1, M_2 > 0$), we could write
	\begin{align*}
		\|\gamma_k(t)-\gamma_k(0)\|&=\|x^k+td_k+t^2(s_k-d_k)-x^k\|\\&\leq t\|d_k\|+t^2\|s_k-d_k\|\\&\leq tM_1+t^2(M_1+M_2),
	\end{align*}
	and the first condition in \eqref{eq::xuhp1} would be satisfied with the forcing function $y_1(t)=tM_1+t^2(M_1+M_2)$.
	However, the same set of assumptions is not enough to guarantee the second condition in \eqref{eq::xuhp1}. Indeed, it holds
	\begin{equation*}
		\frac{\|\gamma_k'(t)-\gamma_k'(0)\|}{\|\gamma_k'(0)\|}= \frac{\|d_k+2t(s_k-d_k)-d_k\|}{\|d_k\|}=\frac{2t\|s_k-d_k\|}{\|d_k\|}.
	\end{equation*} 
	If we assume by contradiction that $y_2(t)$ exists such that $y_2(t)\to 0$ when $t\to 0$ and, for all $k$ and all $t$,
	$$\frac{\|\gamma_k'(t)-\gamma_k'(0)\|}{\|\gamma_k'(0)\|}\le y_2(t),$$
	we also have $\frac{2t\|s_k-d_k\|}{\|d_k\|}\le y_2(t),$ i.e., 
	$$\frac{\|s_k-d_k\|}{\|d_k\|}\le \frac{y_2(t)}{2t}.$$
	
	Now, we can consider the following sequences: $\{d_k\}$ such that $\|d_k\| = 1/k$ and $\{s_k\}$ bounded and such that $\|s_k-d_k\|\ge m$ for all $k$, where $m$ is a positive constant independent on $k$.
	We then have for all $t$ and $k$
	$$\frac{y_2(t)}{2t}\ge \frac{m}{1/k}= mk,$$
	i.e., $mk\le \frac{y_2(t)}{2t}$.
	Whatever the value of $\frac{y_2(t)}{2t}$ is, the above inequality will be violated for $k$ sufficiently large. This makes up for a contradiction. 
\end{remark}

\vspace{0.2cm}

To state complexity results for the proposed class of methods, we need to set an additional boundedness assumption on the second sequence of directions involved in the process. 
\begin{assumption}
	\label{ass::ds}
	There exists $c > 0$ such that $\|d_k\| \le c\|\nabla f(x^k)\|$ and $\|s_k\| \le c\|\nabla f(x^k)\|$ for all $k \in \{0,1,\ldots\}$.
\end{assumption}
Note that the above assumption concerns both directions $d_k$ and $s_k$, but the first inequality is already guaranteed to hold for some value $c$ by the gradient-related hypothesis made on the sequence $\{d_k\}$.
This new assumption allows us to state a preliminary result on the value of the objective function along the curvilinear path which will be crucial for setting bounds on the computational effort. 

\begin{proposition}
	\label{prop::L-smooth}
	Let $x^k \in \mathbb{R}^n$, $f:\mathbb{R}^n\to\mathbb{R}$ be an $L$--smooth function, i.e., $\|\nabla f(x_1) - \nabla f(x_2)\| \le L\|x_1 - x_2\|$ for all $x_1, x_2 \in \mathbb{R}^n$, and $\gamma_k: [0, 1] \to \mathbb{R}^n$ be a curve defined as in \eqref{eq::gengamma2}. Let Assumption \ref{ass::ds} hold. Then, the function $\varphi_k: [0, 1]\to\mathbb{R}$ defined as $\varphi_k(t) = f(\gamma_k(t))$ is $L_k$--smooth in the interval $[0, 1]$ with $L_k = (4c + 37c^2L)\|\nabla f(x^k)\|^2$.
\end{proposition}
\begin{proof}
	By definition of $\gamma_k$, we know that $\gamma'_k(t) = d_k + 2 t (s_k-d_k)$. We can then write
	\begin{gather*}
		\begin{aligned}
			|\varphi'_k(t_1) - \varphi'_k(t_2)| &= |\nabla f(\gamma_k(t_1))^\top\gamma'_k(t_1)-\nabla f(\gamma_k(t_2))^\top\gamma'_k(t_2)|\\&
			=|\nabla f(\gamma_k(t_1))^\top(d_k+2t_1(s_k-d_k))\\&\qquad-\nabla f(\gamma_k(t_2))^\top(d_k+2t_2(s_k-d_k))|\\&=
			|(\nabla f(\gamma_k(t_1))-\nabla f(\gamma_k(t_2)))^\top d_k\\&\qquad+2(s_k-d_k)^\top(t_1\nabla f(\gamma_k(t_1))-t_2\nabla f(\gamma_k(t_2)))|\\&
			\le \|\nabla f(\gamma_k(t_1))-\nabla f(\gamma_k(t_2))\|\|d_k\|\\&\qquad+2(\|s_k\|+\|d_k\|)\|t_1\nabla f(\gamma_k(t_1))-t_2\nabla f(\gamma_k(t_2))\|,
		\end{aligned}
	\end{gather*}
	where the last inequality comes from the triangle inequality.
	Now, we can bound the quantities in the last expression as follows:
	\begin{itemize}
		\item by Assumption \ref{ass::ds}, $\|d_k\|\le c\|\nabla f(x^k)\|$ and $\|s_k\|\le c\|\nabla f(x^k)\|$ for all $k$;
		\item we can obtain an upper bound for $\|\nabla f(\gamma_k(t_1))-\nabla f(\gamma_k(t_2))\|$, i.e.,
		\begin{gather*}
			\begin{aligned}
				\|\nabla f(\gamma_k(t_1))-\nabla f(\gamma_k(t_2))\|&\le L\|\gamma_k(t_1)-\gamma_k(t_2)\|\\&=L\|(t_1-t_2)d_k+(t_1^2-t_2^2)(s_k - d_k)\|\\&=
				L\|(t_1-t_2)(d_k+(t_1+t_2)(s_k - d_k))\|\\&
				\le L|t_1-t_2|(\|d_k\|+|t_1+t_2|(\|s_k\| + \|d_k\|))\\&\le
				L|t_1-t_2|(c\|\nabla f(x^k)\|+4c\|\nabla f(x^k)\|)\\&=
				5cL\|\nabla f(x^k)\||t_1-t_2|;
			\end{aligned}
		\end{gather*}
		\item an upper bound can be found for $\|t_1\nabla f(\gamma_k(t_1))-t_2\nabla f(\gamma_k(t_2))\|$:
		\begin{gather*}
			\begin{aligned}
				\|t_1\nabla f(\gamma_k(t_1))&-t_2\nabla f(\gamma_k(t_2))\|\\&=
				\|t_1\nabla f(\gamma_k(t_1))-t_2\nabla f(\gamma_k(t_2))\\&\qquad+t_1\nabla f(\gamma_k(t_2))-t_1\nabla f(\gamma_k(t_2))\|\\&\le
				|t_1-t_2|\|\nabla f(\gamma_k(t_2))\|\\&\qquad+|t_1|\|\nabla f(\gamma_k(t_1))-\nabla f(\gamma_k(t_2))\|\\&\le
				|t_1-t_2|\|\nabla f(\gamma_k(t_2))+\nabla f(x^k)-\nabla f(x^k)\|\\&\qquad+5cL|t_1|\|\nabla f(x^k)\||t_1-t_2|\\&\le
				|t_1-t_2|(\|\nabla f(x^k)\|+\|\nabla f(\gamma_k(t_2))-\nabla f(x^k)\|)\\&\qquad+5cL\|\nabla f(x^k)\||t_1-t_2|\\&\le
				|t_1-t_2|(\|\nabla f(x^k)\|+L\|\gamma_k(t_2)-x^k\|)\\&\qquad+5cL\|\nabla f(x^k)\||t_1-t_2|
				\\&=
				|t_1-t_2|(\|\nabla f(x^k)\|+L\|t_2d_k+t_2^2(s_k-d_k)\|)\\&\qquad+5cL\|\nabla f(x^k)\||t_1-t_2|\\&\le
				|t_1-t_2|(\|\nabla f(x^k)\|+L(|t_2|\|d_k\|+t_2^2(\|s_k\|+\|d_k\|)))\\&\qquad+5cL\|\nabla f(x^k)\||t_1-t_2|\\&\le
				|t_1-t_2|(\|\nabla f(x^k)\|+3cL\|\nabla f(x^k)\|)\\&\qquad+5cL\|\nabla f(x^k)\||t_1-t_2|\\&
				=(1+3cL)\|\nabla f(x^k)\||t_1-t_2|\\&\qquad+5cL\|\nabla f(x^k)\||t_1-t_2|\\&=
				(1+8cL)\|\nabla f(x^k)\||t_1-t_2|,
			\end{aligned}
		\end{gather*}
		where the third and sixth inequalities come from $t_1, t_2 \in [0, 1]$.
	\end{itemize}
	Putting everything together, we finally get the thesis:
	\begin{gather*}
		\begin{aligned}
			|\varphi'_k(t_1)-\varphi'_k(t_2)|&\le 5c^2L\|\nabla f(x^k)\|^2|t_1-t_2|+4c(1+8cL)\|\nabla f(x^k)\|^2|t_1-t_2|\\&
			=(4c + 37c^2L)\|\nabla f(x^k)\|^2|t_1-t_2|.
		\end{aligned}
	\end{gather*}
\end{proof}

\noindent The descent lemma \cite{bertsekas1999nonlinear} then holds for the $L_k$--smooth function $\varphi_k: [0, 1]\to\mathbb{R}$: 
\begin{equation}
	\label{eq::descent}
	\varphi_k(t) \le \varphi_k(0) + t\varphi'_k(0) + t^2\frac{L_k}{2}, \qquad \forall t \in [0, 1].
\end{equation}

\begin{proposition}
	\label{prop::delta_low}
	Let $f:\mathbb{R}^n\to\mathbb{R}$ be an $L$--smooth function. Let $\{(x^k, \gamma_k)\}$ be the sequence generated by an iterative algorithm of the form \eqref{eq::CS} such that: for all $k$, $\gamma_k$ is defined as in \eqref{eq::gengamma2} and hypotheses i)-ii)-iii) of Proposition \ref{prop:globalconvm} and Assumption \ref{ass::ds} hold. Then, for all $k$, the Armijo condition \eqref{eq::ARcond} is satisfied for all $\alpha\in[0,\Delta_{low}]$, with $\Delta_{low}=\frac{2c_1(1-\sigma)}{4c+37c^2L}$.
\end{proposition}
\begin{proof}
	Let us assume that, at iteration $k$, a step $t$ does not satisfy the curvilinear Armijo condition, i.e., $\varphi_k(t)-\varphi_k(0)>\sigma t \varphi'_k(0)$. By Proposition \ref{prop::L-smooth} and equation \eqref{eq::descent}, we can also write $\varphi_k(t)-\varphi_k(0)\leq t \varphi_k'(0)+t^2\frac{L_k}{2}$. Thus, combining the two inequalities, we get $\sigma t \varphi'_k(0)<t \varphi_k'(0)+t^2\frac{L_k}{2}$ and, rearranging and dividing by $t$, $(1-\sigma)\varphi_k'(0)+t\frac{L_k}{2}>0$.
	
	Now, recalling that, by definition of $\varphi_k$ and Assumption \ref{ass::defcurve}, we have $\varphi_k'(0)=\nabla f(x^k)^\top\gamma'_k(0)=\nabla f(x^k)^\top d_k$, we can rewrite the last result as
	\begin{equation*}
		0 < (1-\sigma)\nabla f(x^k)^\top d_k+t\frac{L_k}{2} \le -c_1(1-\sigma)\|\nabla f(x^k)\|^2+t\frac{L_k}{2},
	\end{equation*}
	where the last inequality comes from the gradient-related conditions and the fact that $\sigma \in (0, 1)$.
	Thus, by definition of $L_k$ (Proposition \ref{prop::L-smooth}), we finally obtain
	\begin{equation*}
		t>\frac{2c_1(1-\sigma)\|\nabla f(x^k)\|^2}{L_k} = \frac{2c_1(1-\sigma)}{4c+37c^2L},
	\end{equation*}
	which completes the proof.
\end{proof}

We are finally able to state the main complexity result.

\begin{proposition}
	Let $f:\mathbb{R}^n\to\mathbb{R}$ be an $L$--smooth function and Assumption \ref{ass::L0} hold. Let $\{(x^k, \gamma_k)\}$ be the sequence generated by an iterative algorithm of the form \eqref{eq::CS} such that the hypotheses and assumptions of Proposition \ref{prop::delta_low} hold. Then, for any $\epsilon > 0$, at most $k_\epsilon$ iterations are needed to produce an iterate $x^k$ such that $\|\nabla f(x^k)\| < \epsilon$, with
	\begin{equation*}
		k_\epsilon \le \frac{f(x^0) - f^\star}{\sigma c_1 \min\{\Delta_0, \delta\Delta_{low}\}\epsilon^2} = \mathcal{O}(\epsilon^{-2}),
	\end{equation*}
	where $f^\star$ is a lower bound of function $f$.
\end{proposition}
\begin{proof}
	The proof is analogous to the one of \cite[Theorem 2.2.2]{cartis2022evaluation}, taking into account Propositions \ref{prop::L-smooth}-\ref{prop::delta_low} and that $f^\star$ is well-defined by Assumption \ref{ass::L0}.
\end{proof}

\begin{remark}
	Note that, by Proposition \ref{prop::delta_low}, there exists a constant interval of stepsize values along the curves for which the Armijo-type condition is satisfied. As a result, the number of backtracks at each iteration is bounded. It then follows immediately that the complexity bound $\mathcal{O}(\epsilon^{-2})$ also applies to the total number of function evaluations. The same bound holds for gradient evaluations for all those methods that, at each iteration, compute the gradient only at the current point.
\end{remark}

\section{The case of heavy-ball}
\label{sec::heavyball}

As anticipated in the Introduction, in this work we are particularly interested in exploiting the framework \eqref{eq::CS} with curves of the form \eqref{eq::gengamma2} as a way to properly induce global convergence properties for Polyak's heavy-ball method~\cite{polyak1964some}. We will thus consider $s_k$ as the heavy-ball direction, that is
\begin{equation}
	\label{eq::hb_sk}
	s_k = -\alpha\nabla f(x^k) + \beta(x^k-x^{k-1}),
\end{equation}
where $\alpha,\beta\in\mathbb{R}_+$ are fixed positive parameters.

The standard heavy ball method at each iteration then always accepts the update 
\begin{equation}\label{HB0}
	x^{k+1} = x^k +s_k,
\end{equation}
without carrying out any type of line-search or other types of controls.

Both in \cite{polyak1964some} and \cite{polyak1987introduction} the convergence of heavy-ball method is proved for a twice continuously differentiable, strongly convex objective function with a Lipschitz continuous gradient, as is stated in the following theorem.
\begin{proposition}[{\cite{polyak1987introduction}}]\label{HB2condiffSCGLC}
	Let $x^\star$ be a nonsingular minimum point of a twice differentiable function $f:\mathbb{R}^n\to\mathbb{R}$. Then, for
	\begin{equation*}
		0\leq\beta\leq1,\quad0<\alpha<\frac{2(1+\beta)}{L},\quad \mu I_n\preceq \nabla^2 f(x^\star)\preceq LI_n,
	\end{equation*}
	we can find an $\epsilon>0$ such that for any $x^0,x^1$, $\|x^0-x^\star\|\leq\epsilon,\|x^1-x^\star\|\leq\epsilon,$ method (\ref{HB0})  converges to $x^\star$ with the rate of geometric progression:
	\begin{equation*}
		\|x^k-x^\star\|\leq c(\bar\delta)(q+\bar\delta)^k,\quad 0\leq q<1,\quad 0<\bar\delta<1-q.
	\end{equation*}
	The quantity $q$ is minimal and equal to
	\begin{equation}
		\label{eq::abstar}
		q^\star=\frac{\sqrt{L}-\sqrt{\mu}}{\sqrt{L}+\sqrt{\mu}}\quad\text{for}\quad \alpha^\star=\frac{4}{(\sqrt{L}+\sqrt{\mu})^2},\quad \beta^\star=\Big(\frac{\sqrt{L}-\sqrt{\mu}}{\sqrt{L}+\sqrt{\mu}}\Big)^2.
	\end{equation}
\end{proposition}
In the most general case, Proposition \ref{HB2condiffSCGLC} proves the local convergence of the heavy-ball method; however, for a quadratic function $f$ the method is globally convergent \cite{ghadimi2015global,polyak1964some,saab2022adaptive}.

Now, one issue concerns the globalization of this algorithm in nonconvex settings. Clearly, the vanilla algorithm needs some adjustment, even just because we have no hint about good values for $\alpha^\star$ and $\beta^\star$ in this scenario. Strategies from the literature to address this problem mainly include:
\begin{itemize}
	\item shrinking the tentative value of $\beta$ until $s_k$ becomes a descent direction \cite{fan2023msl};
	\item determining $\alpha$ and $\beta$ according to some model, so as to get a gradient-related direction, as is done in \cite{lapucci2024globallyconvergentgradientmethod,Lee17,Liu2024,Tang24,zhang2023drsomdimensionreducedsecondorder};
	\item using standard safeguarding strategies \cite{Powell1977}, checking whether $s_k$ is a descent (gradient-related) direction; if not, restarting with a gradient descent step.
\end{itemize}
All of these approaches have two elements in common:
\begin{itemize}
	\item the ``pure'' heavy-ball update is interpreted as the construction of a direction and it is modified if it does not define a descent direction; 
	\item once the direction is acceptable, a line search is carried out.
\end{itemize}
We argue that these two characteristics are typical of a procedure which is closer to a conjugate-gradient type method rather than a globalized version of Polyak's algorithm: a line search is carried out along a descent direction obtained as a suitable combination of the negative gradient and the previous direction.

In fact, Polyak's heavy-ball method for strongly convex functions is not designed to necessarily make updates along descent directions. Applying any of the above globalization strategies may result in the rejection of pure heavy-ball steps in scenarios where the vanilla algorithm is proven to converge.

In this perspective, our proposed globalization strategy is different, as it always checks first the objective value at the heavy-ball update point and, then, backtracks along the curve only if the update is not acceptable. In other words, globalization by means of curve searches might allow to accept the sequence produced by the vanilla algorithm (retrieving the fast linear rate of convergence) in strongly convex problems where safeguarded approaches based on line searches would apply perturbations. For this reason, we argue that the proposed approach represents a proper globalization strategy of heavy-ball method.

In Section \ref{subsec::strongly_convex}, we will provide some numerical insights on the behavior of the different globalization strategies in strongly convex cases.

\section{Computational experiments}
\label{sec::comp-exp}

In this section, we present preliminary results evaluating the effectiveness and efficiency of our approach on a set of benchmark problems. The code\footnote{The implementation code of our approach can be found at \href{https://github.com/dfede3/cs_hb}{github.com/dfede3/cs\_hb}.} for the experimentation was written in \texttt{Python3}. All experiments were conducted on a computer with the following characteristics: Ubuntu 22.04 OS, Intel Xeon Processor E5–2430 v2 (6 cores, 2.50 GHz), 32 GB of RAM.

Throughout this section, we refer to our approach of the form \eqref{eq::CS} as \texttt{CS}, and we denote its non-monotone variant as \texttt{CS\_NMT}. For both variants, we considered the descent direction $d_k = -g_f\nabla f(x^k)$ and the heavy-ball one $s_k$ given in \eqref{eq::hb_sk}, where $g_f$ is a parameter that controls the magnitude of the curve initial velocity (see also Remark \ref{rem::bc}). In our experimental evaluation, \texttt{CS} is compared against the main related approaches (also discussed in Section \ref{sec::heavyball}) to solve problem \eqref{OP}: Gradient Descent (\texttt{GD}); Polyak's heavy-ball method \cite{polyak1964some,polyak1987introduction}, denoted as \texttt{M\_HB}; a variant of the heavy-ball method proposed in \cite{Powell1977,fan2023msl}, referred to as \texttt{M\_RES}, which employs a restart mechanism whenever the heavy-ball direction is not a descent direction; the heavy-ball approach with adaptive momentum \cite{fan2023msl}, denoted as \texttt{M\_BETA}, where the parameter $\beta$ is repeatedly halved until a descent direction is obtained. The stopping criterion, shared by all algorithms, is $\|\nabla f(x^k)\|_\infty\le \varepsilon = 10^{-3}$; moreover, if an algorithm hits 5000 iterations, or it exceeds a time limit of 2 minutes, we consider it a failure. Based on preliminary tuning experiments (not reported here for brevity), we set the algorithm parameters as follows: $g_f = 0.125$, $\alpha = 1$, $\beta = 0.9$, $\Delta_0 = 1$, $\sigma = 10^{-7}$ and $\delta = 0.5$.

\begin{table}[htb]
	\centering
	\scriptsize
	\renewcommand{\arraystretch}{1.25}
	\caption{Benchmark of \texttt{CUTEst} problems used for experimentation.}
	\label{tab::problems}
	\begin{tabular}{|c|c|}
		\hline
		$n$ & \textit{Problems} \\
		\hline
		\hline
		$10$& HILBERTB, STRTCHDV, STRTCHDVB, TRIGON1, TRIGON1B \\
		\hline
		$21$& SANTALS \\
		\hline
		$25$& HATFLDC, HATFLDGLS\\
		\hline
		$46$& BQPGABIM\\
		\hline
		$50$& BQPGASIM, TOINTGOR, TOINTPSP\\
		\hline
		$57$& BA-L1LS, BA-L1SPLS\\
		\hline
		$51$& DECONVB, DECONVU\\
		\hline
		$98$& LUKSAN13LS, LUKSAN14LS\\
		\hline
		$100$& LUKSAN15LS, MANCINO, QING, SENSORS\\
		\hline
		$102$& SPIN2LS\\
		\hline
		$200$& ARGLINA, BROWNAL, VARDIM\\
		\hline
		$380$& HADAMALS\\
		\hline
		$494$& MNISTS0LS, MNISTS5LS\\
		\hline
		$500$& OSCIPATH\\
		\hline
		$502$& INTEQNELS\\
		\hline
		$1000$& EG2, EXTROSNB, KSSLS, POWELLBC\\
		\hline
		$1024$& MSQRTALS, MSQRTBLS\\
		\hline
		$1999$& LINVERSE\\
		\hline
		$2000$& CHARDIS0, EDENSCH\\
		\hline
		\multirow{2}{*}{$3000$}& DIXMAANB, DIXMAANC, DIXMAANF, DIXMAANG, DIXMAANH, DIXMAANJ,\\
		&DIXMAANK, DIXMAANL, DIXMAANN, DIXMAANO, DIXMAANP\\
		\hline
		$4000$& WOODS\\
		\hline
		$4999$& SPMSRTLS\\
		\hline
		\multirow{3}{*}{$5000$}& ARWHEAD, BDEXP, BIGGSB1, BROYDN7D, BROYDNBDLS, BRYBND,\\
		&CRAGGLVY, DQDRTIC, FLETBV3M, NCB20B, NONDQUAR, PENTDI,\\
		&POWELLSG, SCHMVETT, SROSENBR, TOINTGSS, TQUARTIC, VAREIGVL\\
		\hline
		$5010$& NCB20\\
		\hline
		$5625$& FMINSURF\\
		\hline
	\end{tabular}
\end{table}

We conducted our experiments primarily on the benchmark problems listed in Table~\ref{tab::problems}, which are drawn from the \texttt{CUTEst} collection \cite{gould2015cutest}. The problem size $n$ ranges from 10 to 5625. For each problem, all algorithms were initialized with the same feasible starting point, as provided by the \texttt{CUTEst} library.

To assess the performance of the algorithms, we considered the final objective function value $f^\star$ achieved by each method and the corresponding execution time $T$. For a compact visualization of the results, we made use of performance profiles \cite{dolan2002benchmarking}.

\subsection{The strongly convex case}
\label{subsec::strongly_convex}

Before presenting the results on the set of (non-convex) \texttt{CUTEst} problems, we first illustrate the behavior of the tested algorithms on a simple, strongly-convex example. We are particularly interested in validating our claim that, differently than other globalization strategies, the one based on curve searches does not necessarily lead to abandoning the provably good behavior of pure (optimal) heavy-ball. The considered problem is formulated as follows:
\begin{equation}
	\label{eq::logistic}
	\min_{x\in\mathbb{R}^n}\log\left(1 + e^{c^\top x}\right) + \frac{1}{2}\|x\|^2,
\end{equation}
where $c = [c_1, c_2]^\top = [34, -1]^\top$. In this case, we note that the strong convexity constant is $\mu = 1$, and the Lipschitz constant of the gradient is given by $L = \frac{1}{4}(c_1^2 + c_2^2) + 1$. The problem admits a unique optimal solution, which is $x^\star = [-0.158, 0.005]^\top$. Since this is a smooth and strongly convex problem, the optimal parameters $\alpha^\star$ and $\beta^\star$ for the standard heavy-ball method can be computed as indicated in \eqref{eq::abstar}. In Figure~\ref{fig:strong_convex}, we compare the iterative convergence behavior of \texttt{CS}, \texttt{CS\_NMT} ($M = 20$), \texttt{GD}, \texttt{M\_HB}, and \texttt{M\_RES}, in terms of their distance to the optimal solution.

\begin{figure}[!h]
	\centering
	\includegraphics[width=\linewidth]{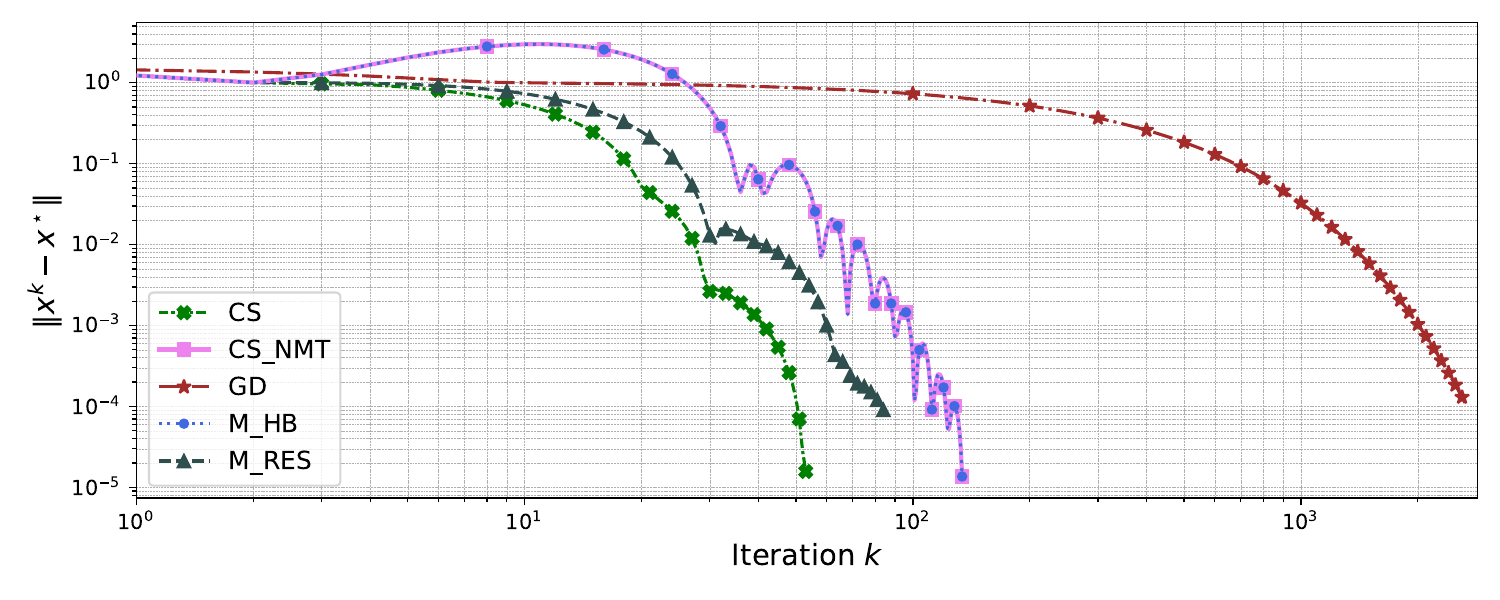}
	\caption{Plot of the distance between the current solution $x^k$ and the optimal solution $x^\star$ of problem \eqref{eq::logistic} across the iterations for \texttt{CS}, \texttt{CS\_NMT}, \texttt{GD}, \texttt{M\_HB}, \texttt{M\_RES}. Both axes are displayed on a logarithmic scale.}
	\label{fig:strong_convex}
\end{figure}

We can observe that, in fact, the nonmonotone curve search based version \texttt{CS\_NMT} ends up perfectly retracing the steps of pure heavy-ball \texttt{M\_HB}, which, as expected from theory, outperforms simple gradient descent \texttt{GD}. Interestingly, the strategy based on restart \texttt{M\_RES}, while abandoning the heavy-ball path, allows for faster convergence to the solution. However, the version with monotone curve search \texttt{CS} performs even better, proving to be the most efficient among the methods tested on this problem.

\subsection{\texttt{CUTEst} problems}

We conclude the experimental section by comparing the approaches on the set of \texttt{CUTEst} problems listed in Table \ref{tab::problems}. For a compact representation of the results, Figure \ref{fig:pp} shows the
performance profiles of \texttt{CS}, \texttt{GD}, \texttt{M\_HB}, \texttt{M\_RES} and \texttt{M\_BETA} in terms of $f^\star$ and $T$. To better highlight the differences among the methods, we present separate plots for each performance metric: Figures~\ref{fig:Tall}-\ref{fig:fall} report results for the full set of problems; Figures~\ref{fig:Tlow}-\ref{fig:flow} focus on problems with $n < 1000$; Figures~\ref{fig:Thigh}-\ref{fig:fhigh} show results for problems with $n \ge 1000$.

\begin{figure}
	\centering
	\subfloat[\label{fig:Tall}$T$, all problems.]{\includegraphics[width=0.32\textwidth]{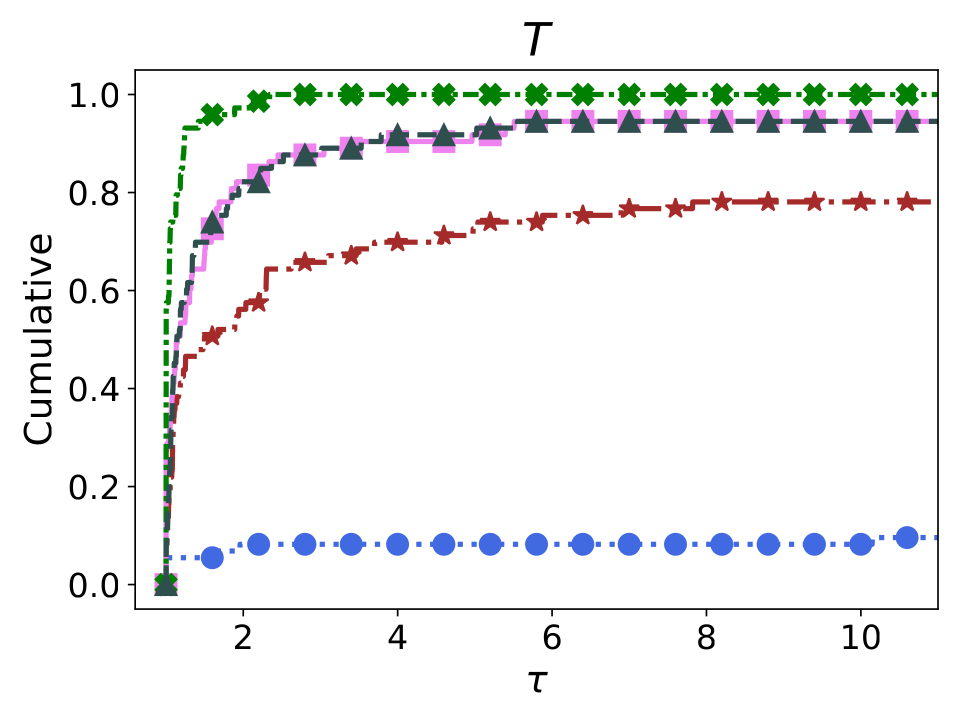}}
	\hfil
	\subfloat[\label{fig:Tlow}$T$, $n < 1000$.]{\includegraphics[width=0.32\textwidth]{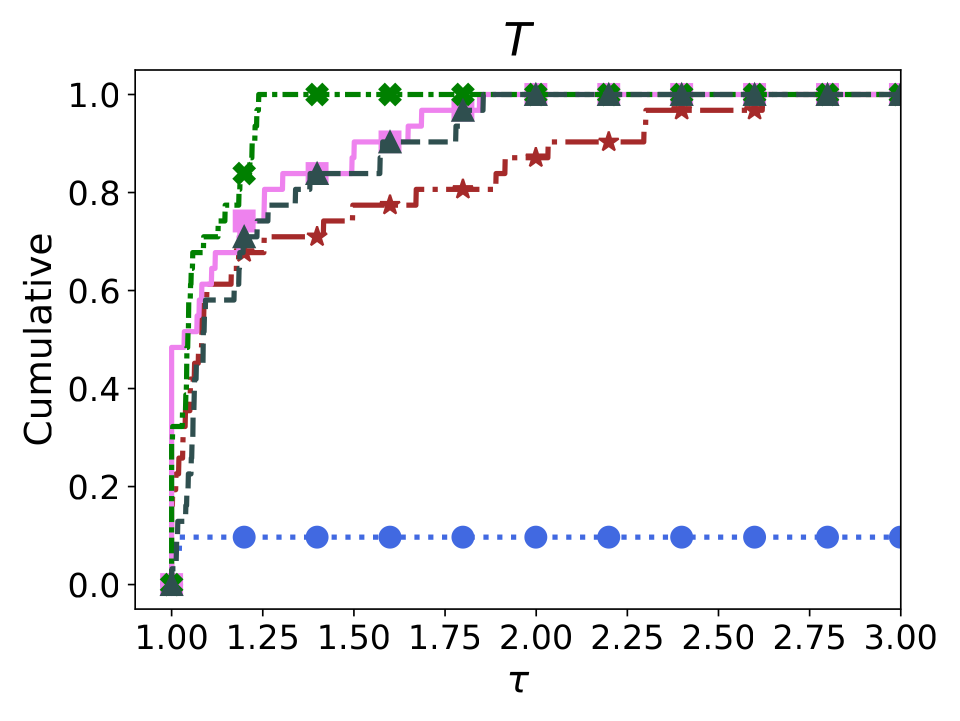}}
	\hfil
	\subfloat[\label{fig:Thigh}$T$, $n \ge 1000$.]{\includegraphics[width=0.32\textwidth]{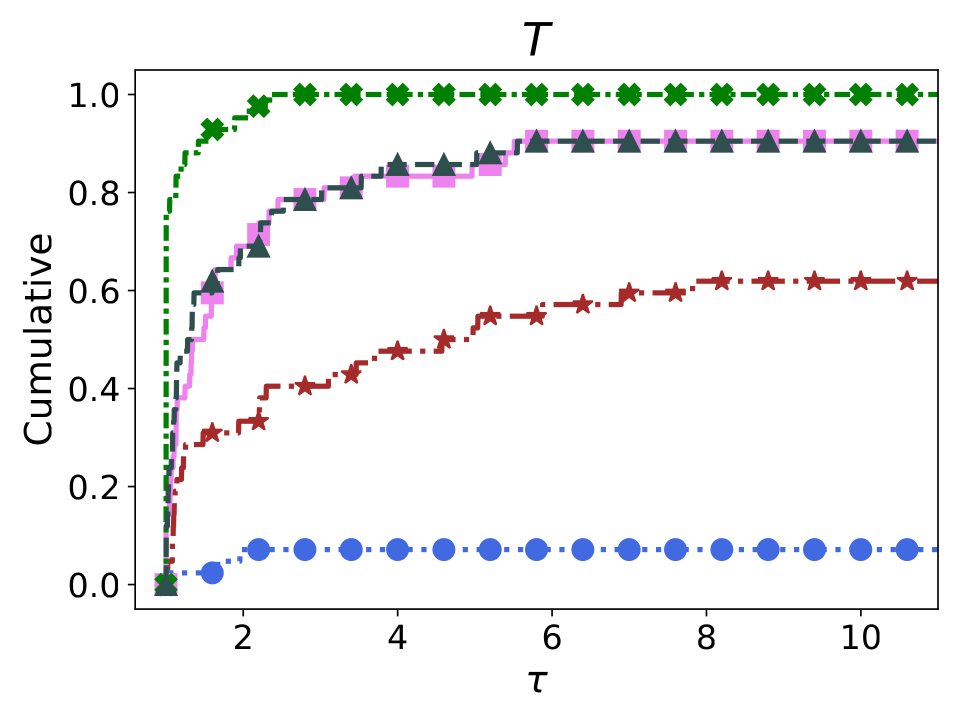}}
	\\
	\subfloat[\label{fig:fall}$f^\star$, all problems.]{\includegraphics[width=0.32\textwidth]{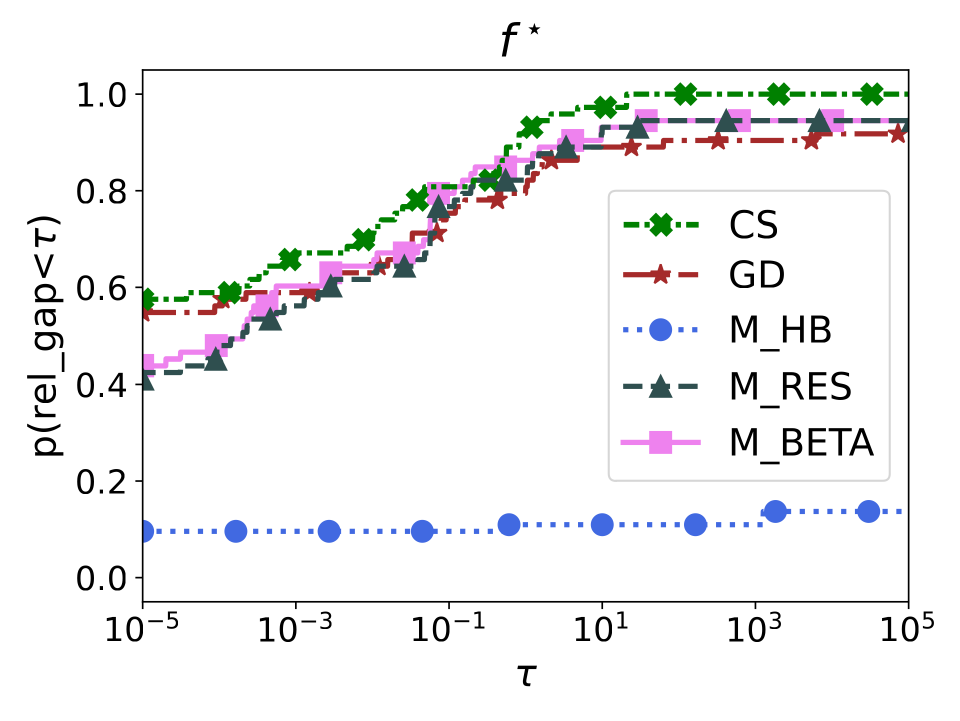}}
	\hfil
	\subfloat[\label{fig:flow}$f^\star$, $n < 1000$.]{\includegraphics[width=0.32\textwidth]{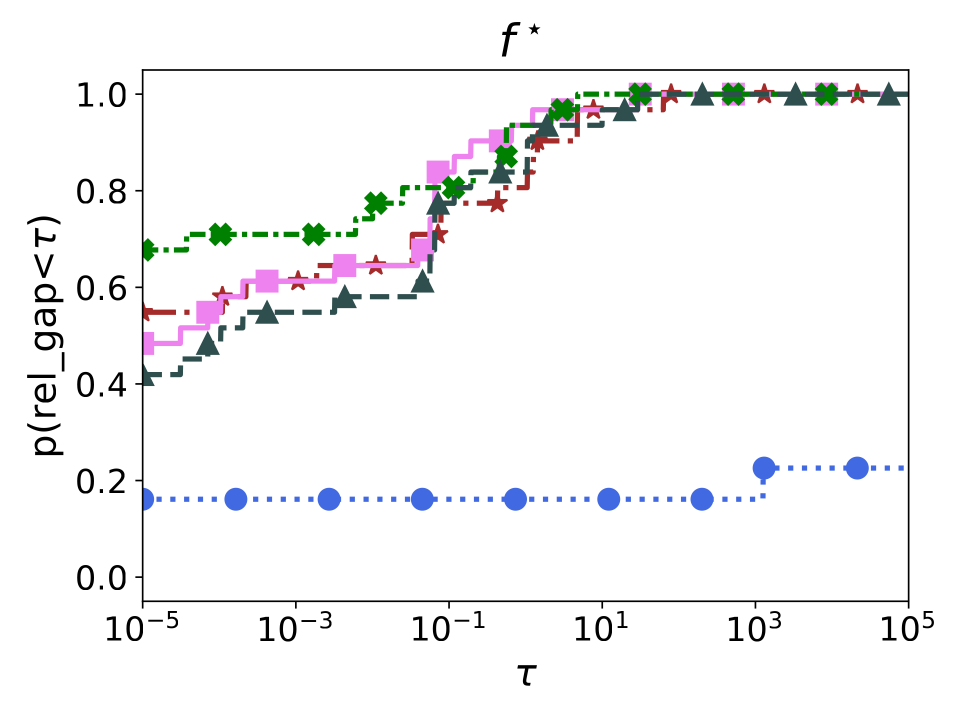}}
	\hfil
	\subfloat[\label{fig:fhigh}$f^\star$, $n \ge 1000$.]{\includegraphics[width=0.32\textwidth]{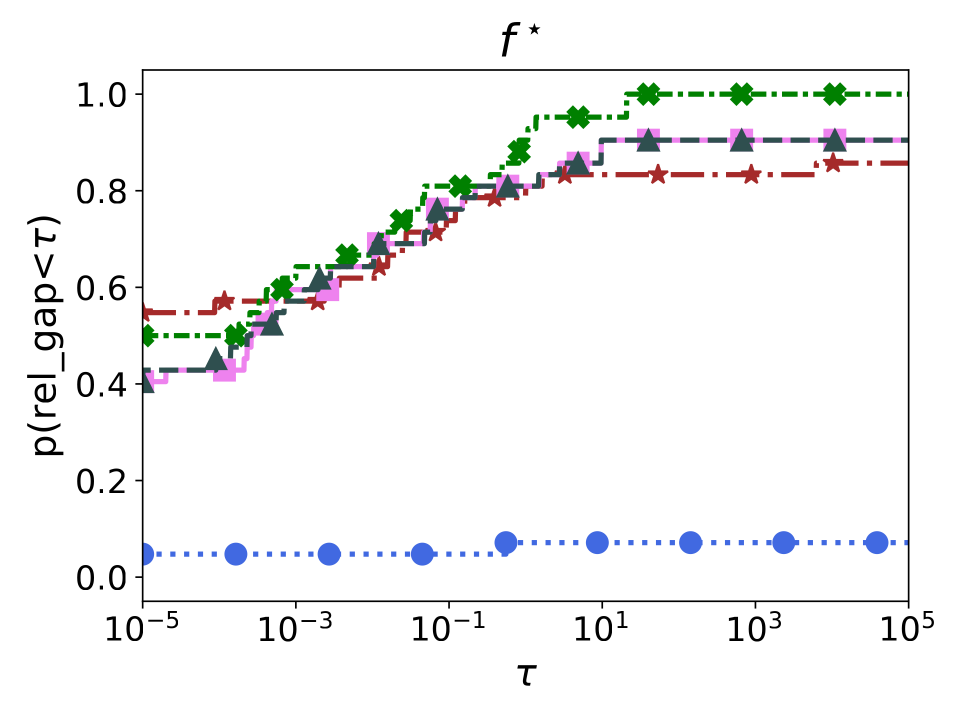}}
	\caption{Performance profiles in terms of $f^\star$ and $T$ obtained by \texttt{CS}, \texttt{GD}, \texttt{M\_HB}, \texttt{M\_RES} and \texttt{M\_BETA} on the \texttt{CUTEst} problems listed in Table \ref{tab::problems}.}
	\label{fig:pp}
\end{figure}

In this experimental setting, the least effective methodology was \texttt{M\_HB}, which clearly struggles due to the challenging selection of $\alpha$ and $\beta$ in non-convex scenarios. However, we can observe that the introduction of a momentum term, if done properly, allows to significantly boost the performance of standard \texttt{GD} both in terms of efficiency and effectiveness.
The two globalization strategies employed in \texttt{M\_RES} and \texttt{M\_BETA} seem to lead to quite similar results, whereas our \texttt{CS} approach apparently leads to further improvements, especially as the size of the problem grows. This behavior aligns with findings in the existing literature on curve search methods (see, e.g., \cite{xu2016global}). Interestingly, Figures \ref{fig:fall}-\ref{fig:fhigh} ensure us that the improvement in efficiency is not due to the algorithm ending up in worse stationary points w.r.t.\ the other algorithms.  

\section{Conclusions}
\label{sec::conclusions}

In this work, we proposed a new curve search based globalization strategy for heavy-ball type methods for unconstrained optimization problems. Our framework, which is in fact applicable to a broad class of gradient-based iterative methods, ensures global convergence under reasonable assumptions by employing (possibly nonmonotone) backtracking searches along carefully designed search curves. To the best of our knowledge, this is the first work to provide convergence results for a nonmonotone curve search strategy. Moreover, another original contribution consists in the derivation of worst-case complexity results in the nonconvex setting, established under mild conditions for quadratic search curves.

From a numerical perspective, we demonstrated that, unlike other globalization strategies from the literature, our approach retains the desirable behavior of the pure heavy-ball method in the strongly convex case. Finally, extensive computational experiments on nonconvex problems confirmed the practical effectiveness of our method, which outperformed existing alternatives in terms of efficiency.

\renewcommand{\theequation}{A\arabic{equation}}
\setcounter{equation}{0}

\renewcommand{\thetable}{C\Roman{table}}
\setcounter{table}{0}
\renewcommand{\theHtable}{Supplement\thetable}

\renewcommand{\theproposition}{A\arabic{proposition}}
\setcounter{proposition}{0}
\renewcommand{\theassumption}{A\arabic{assumption}}
\setcounter{assumption}{0}

\renewcommand{\thealgocf}{A\arabic{algocf}}
\setcounter{algocf}{0}

{\appendices
	\section{Proof of Proposition \ref{prop::nm_globalconvm_strong}}
	\label{app::proof}
	
	\begin{proof}
		We follow the proof of Proposition \ref{prop::nm_globalconvm} until we get equation \eqref{eq::for_app}:
		\begin{equation}
			\label{eq::from_app}
			\begin{aligned}
				f(x^{l(k)}) &\le f(x^{l(l(k)-1)})-\sigma c_1t_{l(k)-1}\|\nabla f(x^{l(k)-1})\|^2 \\&\le f(x^{l(l(k)-1)})-\sigma \frac{c_1}{c_2^2}t_{l(k)-1}\|d_{l(k)-1}\|^2,
			\end{aligned}
		\end{equation}
		where the last inequality comes from the second gradient-related condition.
		
		The inequality holds for every $k$; hence, we can take the limit for $k\to\infty$ recalling $\lim_{k\to\infty}f(x^{l(k)})=f^\star$ (see proof of Proposition \ref{prop::nm_globalconvm}), $\sigma \in (0, 1)$, $c_1 > 0$, $c_2 > 0$ and the non-negativeness of the norm; we get 
		\begin{equation}
			\label{eq::lim_bef_ind}
			\lim_{k\to\infty}t_{l(k)-1}\|d_{l(k)-1}\|=0.
		\end{equation}
		
		Now, let us prove that $\lim_{k\to\infty}t_k\|\nabla f(x^k)\|=0$. Let $\hat{l}(k)=l(k+M+2)$. First, we show by induction that, for any $j\geq1$,
		\begin{equation}
			\label{eq::th_ind}
			\lim_{k\to\infty}t_{\hat{l}(k)-j}\|d_{\hat{l}(k)-j}\|=0, \qquad \lim_{k\to\infty} f(x^{\hat{l}(k)-j}) = \lim_{k\to\infty} f(x^{l(k)}),
		\end{equation}
		assuming in the following that $k \ge j - 1$ without loss of generality. If $j=1$, since $\{\hat{l}(k)\}\subset \{l(k)\}$, the first equation in  \eqref{eq::th_ind} easily follows from \eqref{eq::lim_bef_ind}. These two equations, by assumption on $\gamma_k$ for all $k$, also imply that
		\begin{gather*}
			\lim_{k\to\infty}\|x^{\hat{l}(k)}-x^{\hat{l}(k)-1}\| = \lim_{k\to\infty}\|\gamma_{\hat{l}(k)-1}(t_{\hat{l}(k)-1})-\gamma_{\hat{l}(k)-1}(0)\|\le\lim_{k\to\infty}y(t_{\hat{l}(k)-1}\|d_{\hat{l}(k)-1}\|) = 0,
		\end{gather*}
		resulting by the non-negativeness of the norm in $\|x^{\hat{l}(k)}-x^{\hat{l}(k)-1}\| \to 0$. As a consequence, the second equation of \eqref{eq::th_ind} holds for $j=1$ by the continuity of $f$ and the fact that $\{\hat{l}(k)\}\subset \{l(k)\}$.
		
		Then, let us assume that \eqref{eq::th_ind} holds for a given $j$. Similar as for \eqref{eq::from_app}, by the Armijo \eqref{eq::nm_ARcond} and the gradient-related conditions we can get
		\begin{gather*}
			f(x^{\hat{l}(k)-j}) \le f(x^{l(\hat{l}(k)-j-1)}) - \sigma \frac{c_1}{c_2^2} t_{\hat{l}(k)-j-1}\|d_{\hat{l}(k)-j-1}\|^2.
		\end{gather*}
		Taking again the limit for $k\to\infty$ and recalling the second equation in \eqref{eq::th_ind}, we have $\lim_{k\to\infty} t_{\hat l(k)-(j+1)}\|d_{\hat l(k)-(j+1)}\|=0$, which, as above, implies $\|x^{\hat l(k)-j}-x^{\hat l(k)-(j+1)}\|\to0$; consequently, by the second equation in \eqref{eq::th_ind} and the continuity of $f$, $\lim_{k\to\infty} f(x^{\hat l(k)-(j+1)})=\lim_{k\to\infty} f(x^{\hat l(k)-j})=\lim_{k\to\infty} f(x^{l(k)})$. Hence, we have proved that \eqref{eq::th_ind} holds for any given $j\geq1$.
		
		Now, for any $k$, $x^{\hat{l}(k)} - x^{k+1} = \sum_{j=1}^{\hat{l}(k) - k - 1}(x^{\hat{l}(k) - j + 1} - x^{\hat{l}(k) - j})$, and thus, by assumption on $\gamma_k$, for all $k$, and the triangle inequality,
		\begin{gather*}
			\begin{aligned}
				\|x^{\hat{l}(k)} &- x^{k+1}\| \le \sum_{j=1}^{\hat{l}(k) - k - 1}\|x^{\hat{l}(k) - j + 1} - x^{\hat{l}(k) - j}\| \\&= \sum_{j=1}^{\hat{l}(k) - k - 1}\|\gamma_{\hat{l}(k) - j}(t_{\hat{l}(k) - j}) - \gamma_{\hat{l}(k) - j}(0)\| \le \sum_{j=1}^{\hat{l}(k) - k - 1}y(t_{\hat{l}(k) - j}\|d_{\hat{l}(k) - j}\|).
			\end{aligned}
		\end{gather*}
		By definition of $\hat{l}(k)$ and $l(k)$ (see proof of Proposition \ref{prop::nm_globalconvm}), we have $\hat{l}(k) - k - 1 = l(k+M+2) - k - 1 \leq M + 1$; thus, recalling the first equation in \eqref{eq::th_ind}, that $y$ is a forcing function and the norm operator is non-negative, we can take the limit for $k \to \infty$ on the last result to get $\lim_{k \to \infty} \|x^{k+1} - x^{\hat l(k)}\| = 0$. Moreover, since $f$ is continuous and $\{\hat{l}(k)\}\subset \{l(k)\}$, we have $\lim_{k \to \infty} f(x^{k+1}) = \lim_{k \to \infty} f(x^{l(k)})$.
		By \eqref{eq::nm_ARcond} and the gradient-related conditions we have $f(x^{k+1}) \leq f(x^{l(k)}) - \sigma c_1 t_k\|\nabla f(x^k)\|^2$.
		Thus, taking the limit for $k \to \infty$, since $\{f(x^{l(k)})\}$ admits a limit, $\sigma \in (0, 1)$ and $c_1 > 0$, we obtain 
		\begin{equation}
			\label{eq::final_lim}
			\lim_{k \to \infty} t_k\|\nabla f(x^k)\| = 0.
		\end{equation}
		
		Now, let $\bar{x}$ be one of the accumulation points of $x^k$, i.e., there exists $K\subseteq\{0,1,\dots,\}$ such that $\lim_{k\in K,k\to\infty} x^k=\bar{x}$, and let us suppose by contradiction that $\|\nabla f(\bar{x})\| = \nu > 0$. Hence, by \eqref{eq::final_lim} it must be $\lim_{k \in K, k \to \infty} t_k = 0$. We thus follow a similar reasoning as the last part of the proof of Proposition \ref{prop:globalconvm} to get a contradiction, and this completes the proof of the thesis.
	\end{proof}
}

\section*{Acknowledgements}
The authors are very grateful to Dr. T. Trinci for the fruitful discussions.

\section*{Conflict of Interest}
The authors declare that they have no conflict of interest.

\section*{Funding Sources}
This research did not receive any specific grant from funding agencies in the public, commercial, or not-for-profit sectors.

\section*{Data Availability Statement}
Data sharing is not applicable to this article, as no new data were created or analyzed in this study.

\section*{Code Availability Statement}
The implementation code of the approach presented in the paper can be found at \href{https://github.com/dfede3/cs_hb}{github.com/dfede3/cs\_hb}.

\bibliographystyle{abbrv}

\end{document}